\RequirePackage{fix-cm}
\documentclass[smallextended]{svjour3}       
\smartqed  
\usepackage{graphicx}
\usepackage{amssymb,amsmath}
\usepackage[hug,balance,dpi=1200]{diagrams}
\usepackage{mathptmx}      
\newcommand{\Mult}{{\mathcal M}} 

\newcommand{\compose}{\circ}
\renewcommand{\H}{\ensuremath{{ H}}} 
\newcommand{\C}{\mathbb C}

\newcommand{\BH}{\ensuremath{B(\H)}} 
\newcommand{\tensor}{\otimes} 
\newcommand{\mintensor}{\otimes_{min}} 
\newcommand{\HtH}{\ensuremath{\H\tensor\H}}

\newcommand{\adjointable}{\ensuremath{\mathcal L}}
\newcommand{\isom}{\cong}

\newcommand{\antipode}{\ensuremath{\kappa}} 
\renewcommand{\antipode}{\ensuremath{S}} 

\newcommand{\haar}{\tau}

\renewcommand{\haar}{\tau}

\newcommand{\F}{\mbox{\ensuremath{ F}}} 
\newcommand{\inv}{\ensuremath{{}^{-1}}}
\newcommand{\Id}{\mbox{\rm Id}}

\newcommand{\arrow}{\rightarrow}

\newcommand{\coproduct}[1][\empty]{\ensuremath{\varDelta_#1}}

\newcommand{\convolution}{\diamond}   \newcommand{\convolve}{\convolution}
\newcommand{\boxproduct}{\,\ast\,} 
\newcommand{\comment}[1]{\-\marginpar[\raggedleft\footnotesize\it\textcolor{Sienna4}{#1\smallskip}]{\raggedright\footnotesize\it\textcolor{Sienna4}{#1\smallskip}}}
\renewcommand{\comment}[1]{}
\newcommand{\compact}{\ensuremath{ K}}
\newcommand{\ip}[2]{\ensuremath{\left\langle #1,#2\right\rangle}}
\newcommand{\Cu}{\ensuremath{{ C}\hspace{-.7pt}u}}

\journalname{Advances in Operator Theory}
\begin{document}

\title{{Isometries and isomorphism classes of Hilbert modules}\thanks{Thanks for financial support are due to NSERC (Canada), the Fredrik and Catherine Eaton Visitorship (Canada/UK), and IMPAN (Poland).}
}

\titlerunning{[Isomorphism classes of Hilbert modules}        

\author{Dan Z. Ku\v{c}erovsk\'{y}
}


\institute{D.  Ku\v{c}erovsk\'{y} \at
              {University of New Brunswick at Fredericton \newline%
Canada\qquad E3B 5A3} \\
              Tel.: +1-506-458-7364\\
              Fax: + 1-506-453-4705\\
              \email{dkucerov@unb.ca}           
}

\date{Received: date / Accepted: date}

\maketitle

\begin{abstract}
We show that a $A$-linear map of Hilbert $A$-modules is induced by a unitary Hilbert module operator if and only if it extends to an ordinary unitary on appropriately defined enveloping Hilbert spaces.  Applications to the theory of multiplicative unitaries  compute the equivalence classes of Hilbert modules over a class of \textsl{C*}-algebraic quantum groups. We thus develop a theory that for example could be used to show non-existence of certain co-actions. In particular, we show that the Cuntz semigroup functor takes a co-action to a multiplicative action.

\end{abstract}

\section{Introduction}

Hilbert modules have many remarkable properties, and as pointed out by Lance \cite{lance1994} and others, apparently quite weak notions of isometry of Hilbert modules imply isomorphism of Hilbert modules. We apply this basic fact in several different settings. As a guide to possible applications, we consider briefly the Cuntz semigroup and the K-theory group. These sometimes rather technical objects become particularly attractive in the setting of C*-algebraic quantum groups, where we find that there exists a product operation on  Hilbert modules that is quite nice and seems to be distinct from, although related to, the usual interior and exterior tensor products of Hilbert modules.

We then study isomorphism classes of Hilbert modules. These have a natural semigroup structure, under direct sum. Our most advanced result is roughly as follow:

 \begin{theorem}Let the \textsl{C*}-algebraic quantum group $A$ have a left co-action upon $B$ with a faithful invariant state. Then we obtain a well-behaved action of the semiring of isomorphism classes of Hilbert modules over $A$ upon the semigroup of isomorphism classes of Hilbert modules over $B.$\end{theorem}

In the above, by an invariant state for a left co-action $\delta$ of a compact or discrete \textsl{C*}-algebraic quantum group $(A,\coproduct{} )$ on $B$ we mean a state $\eta$ on $B$ such that $(\Id\tensor\eta)\delta=\eta(\cdot)1.$ By a well-behaved action of the semiring we mean a multiplicative action where the product on the ring distributes over the direct sum in the semigroup.
As well as providing some kind of invariant associated with a co-action, it is quite possible that the above can be viewed as an algebraic topology type of obstacle to the existence of certain co-actions. In other words, we could show non-existence of certain co-actions by showing that one of the implied equations for the action of the semiring on the semigroup must fail.

The work of Goswami \cite{Goswami} on non-existence of certain co-actions  provides evidence that the question of nonexistence may be interesting. Our methods are however completely different than Goswami's. It is possible that eventually our work may find application to showing noncommutative Borsuk-Ulam type theorems \cite{BDH2015}, which roughly speaking assert the lack of existence of a co-action on a special kind of \textsl{C*}-algebra, a noncommutative join.

The paper is organized as follows. In section 2 are the main results on Hilbert modules and the details of the already mentioned product operation. In section 3 we compute this product in several cases. In section 4 we apply our theory to the case of K-theory. In section 5 we consider the mostly finite-dimensional theory of Hopf ideals. In section 6 we give our deepest results, pertaining to co-actions.

\section{Hilbert modules and maps}

We recall that the standard Hilbert module $\H_B$ over a C*-algebra $B$ is by definition the set of sequences $(b_i )$ such that $\sum b_i ^* b_i$ converges in the C*-norm on $B.$ It is a subtle\footnote{There is an amusing discussion on page 239 of \cite{WeggeOlsen}
about various pitfalls in finding the right definitions to make this work.}
fact that (see for example, \cite[pg. 34-35]{lance1994}) this is isomorphic to the (exterior) tensor product $B\tensor \H$ where $\H$ denotes the usual countably generated complex Hilbert space.

\begin{proposition} Suppose $A$ is a C*-algebra with faithful state $g$ and that $H_1$ and $H_2$ are (pre)Hilbert modules over $A.$ Let the $A$-module map $W\colon H_1\arrow H_2$ satisfy the property $$ g\left( \ip{x}{y}_{H_1}\right)= g\left(\ip{Wx}{Wy}_{H_2}\right)$$ for all $x$ and $y$ in $H_1.$The map $W$ is then an isometry with respect to the Hilbert module norms.\label{prop:2isometries.want.to.be.isomorphisms}
\end{proposition}
\begin{proof}
Given   $y\in A^{+}$, for $\sigma$-weakly continuous (positive) linear functionals $h$ we have the dual norm formula
$$\|y\|_{A} =\sup\frac{|h(y)|}{\|h\|}.$$ By the Radon--Nicod\'ym theorem \cite{Sakai1965},  it is sufficient to consider linear functionals of the form $h=g(c\cdot c)$ where  $c$ comes from the double dual. It follows that linear functionals of the form $h=g(c\cdot c)$ with $c$ in the algebra are dense within the  class of linear functionals that we took the supremum over.

We thus have that
$$\|y\|_{A}=\sup_{c\not=0} \frac{|g(cyc)|}{\|g(c\cdot c)\|}.$$ The norm of a positive linear functional is equal to the value of the functional at $1,$ so that $\|g(c\cdot c)\|=g(c^2).$

(This could also have been shown using the Kadison transitivity theorem.)

 In any case, we have that
$$\|y\|_{A}=\sup_{c\not=0} \frac{|g(cyc)|}{g(c^2)},$$ where the supremum is over the nonzero positive elements $c$ in $A,$ and $g$ is the given faithful state.

Then we notice that evidently
\begin{eqnarray*}
 g\left(c\ip{x}{x}_{H_1}c\right) & =& g\left( \ip{c x}{cx}_{H_1}\right) \\  &=& g\left(\ip{Wc x}{Wcx}_{H_2}\right) \\  &=& g\left(c\ip{Wx}{Wx}_{H_2}c\right).\end{eqnarray*}
Then it follows from the first part of the proof that$\|\ip{x}{x}_{H_1}\|_{A}=\|\ip{Wx}{Wx}_{H_2}\|_{A}.$

By the polarization identity, $W$ is thus an isometry, with respect to the Hilbert module norms, as was to be shown.
\end{proof}

We now make some remarks leading to a pleasant geometric interpretation of the above lemma.

We remark first that if $A$ is a \textsl{C*-}algebra with a faithful state, $f,$ we may construct the GNS Hilbert space $\H$ associated with it. The usual direct sum over all states can be omitted if we are given a faithful state.
This construction can be generalized somewhat: \label{pageref:GNSconstr} it is sufficient to have   a countably generated Hilbert module over the given C*-algebra $A,$ and then the given faithful state together with the usual GNS construction provides a GNS Hilbert space $\H$ that envelopes the given Hilbert module.  We clarify our statements by giving the construction briefly. Thus, we regard the elements of $E$ as belonging to a pre-Hilbert space with inner product $f(\ip{ e_1}{e_2 })$.
Since $f$ is a faithful state and $\ip{e_1}{e_2}$ is a Hilbert module inner product, the pre-inner product is faithful, and we don't need to take quotients by a subspace of indefinite vectors as in the usual GNS construction. The upshot is that we obtain a Hilbert space $H$ inside which the given Hilbert module $E$ is a subspace. Generally the subspace will not be closed.
 We should clarify that this construction is not the induced representation (see \cite[24]{RW}) construction. Instead of forming a tensor product and building an imprimitivity bimodule, we are simply using the elements of the given Hilbert module $E$ as the elements for the GNS construction.

Now, we combine our lemma with some known results on Hilbert modules \cite[Prop 3.3]{MF}.

\begin{theorem}Let $W\colon H_1 \arrow H_2$ be a surjective $A$-linear map between Hilbert $A$-modules, where $A$ is a C*-algebra with faithful state, $g.$ Then $W$ is a Hilbert module unitary in $\mathcal{L}(H_1,H_2)$ if and only if $W$  has the  property $ g\left(\ip{x}{y}_{H_1}\right)= g\left(\ip{Wx}{Wy}_{H_2}\right)$  for all $x$ and $y$ in $H_1.$                        \label{th:equivalences.of.maps}
\end{theorem}
\begin{proof} The main thing to show is that the given property implies the apparently much stronger Hilbert module unitary property. Our Proposition \ref{prop:2isometries.want.to.be.isomorphisms} shows that the given property implies $W$ is an isometry with respect to Hilbert module norms in the sense that $\|x\|_{H_1}=\|W(x)\|_{H_2}$ . It is known \cite[Theorem 3.2]{blecher}\cite{lance1994}\cite[Prop. 3.3]{MF} that a surjective  $A$-linear Banach module map $F$ of Hilbert $A$-modules that has these properties $F(ax)=aF(x)$ and $\|x\|_{E_1}=\|F(x)\|_{E_2}$ is automatically an isomorphism of Hilbert modules.  Finally, Theorem 1 in the book \cite{lance1994} shows that an isomorphism of Hilbert  $A$-modules is automatically adjointable, and therefore is a unitary operator in the Hilbert module sense.
\end{proof}

Now the promised striking geometrical interpretation of the above facts can be stated:
\begin{corollary} An $A$-linear map of Hilbert $A$-modules is induced by a unitary Hilbert module operator if and only if it extends to an ordinary unitary on the enveloping GNS Hilbert spaces. \end{corollary}

To specify  which faithful state is being used to construct the Hilbert space (using the method discussed on page \pageref{pageref:GNSconstr}) we will say ``enveloping GNS Hilbert space of $\H_A$ with respect to $g$''
or language to a similar effect.

In the setting of C*-algebraic quantum groups, we will observe a situation where a Hilbert space unitary acts on a Hilbert $A$-module but instead of being $A$-linear it has a skew-linear property that will be explained later. This skew linear property will give us the opportunity to make an interesting application of the above theorem.

Actually, there exist in the literature a few different conditions that are ultimately equivalent to isomorphism for Hilbert modules, and we can state a corollary clarifying the differences and rounding out our result by adding a few of these conditions:

\begin{corollary}Let $U$ be an  $A$-linear map surjective between Hilbert $A$-modules. TFAE:
 \begin{enumerate}\item $U$ is a unitary element of the bounded adjointable operators,
\item $U$ is a Hilbert module isomorphism,
\item $U$ is a Hilbert module norm isometry, and
\item $U$ has the property $ \phi\left(\ip{x}{y}_{E_1}\right)= \phi\left(\ip{Ux}{Uy}_{E_2}\right)$.
\end{enumerate}\label{th:equivalences.of.maps}
In the above, $\phi$ is some fixed choice of faithful state on $A.$
\end{corollary}

Hopf algebras are bi-algebras with an antipode map $\antipode.$ See \cite{abe} for information on Hopf algebras in the algebraic setting. If one tries to frame the nice  theory of Hopf algebras in a C*-algebraic setting, there are several slightly different approaches that can be taken.

The theory of Kats and Pal'yutkin \cite{KP1966},  its further development by Enock and Schwartz \cite{EnockSchwartz}; the theory of Woronowicz \cite{woronowicz1987,PW1990}, of Baaj and Skandalis \cite{BS,BSk.Hopf,BBS}, and of Kustermans and Vaes \cite{KustermansVaes} should especially be mentioned.  The presence or absence of multiplicative units makes a difference in the theory, and generally the unital case, also known as the compact case, is simpler.

For example, in the especially attractive Baaj--Skandalis framework of C*-algebraic quantum groups provided by \cite{BS}, a compact C*-algebraic quantum group is unital as an algebra, and has structure maps that are compatible with the C*-algebraic structure.
In this picture, a compact quantum group $(A ,\coproduct{})$ is first of all a unital C$^*$-algebra $A$ with a coproduct map.
The coproduct map is a unital $*$-homomorphism $\coproduct{} \colon A\rightarrow A\otimes A$
such that $(\coproduct{}\otimes\iota)\coproduct{}=(\iota\otimes\coproduct{})\coproduct{}$ and such that
$\coproduct{}(A)(A\otimes1)$ and $\coproduct{}(A)(1\otimes A)$ are dense in $A\otimes
A$.

For any compact quantum group there exists a unique Haar state
$\tau$ of $A$ which is left- and right-invariant, i.e.,
$(\iota\otimes \tau)\coproduct{} =\tau(\cdot)1$ and
$(\tau\otimes\iota)\coproduct{}=\tau(\cdot)1$, respectively. This
functional is called the Haar functional and is very usually assumed to
be faithful.

Multiplicative unitaries were introduced by Baaj and Skandalis in  \cite{BS} (Chapter 4), see also \cite{BBS}.

Given a compact quantum group  $A$ with coproduct $\coproduct{}\colon A\arrow  A \mintensor A,$ and faithful Haar state $\haar\colon A\arrow\C,$ let $\H$ be the Hilbert space $L^2 (A,\haar)$ that comes from the GNS construction applied to $A$ via $\haar.$ Let $e\in \H$  be the image $e:=\pi_{\haar}1$ of $1$ in $\H.$ Then, $Ae$ is dense in $\H$ and there is a multiplicative unitary $V\in\adjointable(\HtH)$ by $$Vae\tensor be=\coproduct{}(a)e\tensor be.$$ A different definition, $W,$ satisfying $$W^* ae\tensor be=\coproduct{}(b)ae\tensor e,$$ was used in a more general setting (the locally compact case) by Kustermans and Vaes \cite{KustermansVaes}, see also \cite{MaesVanDaele}.

The next lemma reminds one of the formula for coproducts in terms of multiplicative unitaries:  $\coproduct{}(a)=V(x\tensor 1)V^*.$ See, \textit{e.g.} \cite[Th. 3.8]{BS}.

\begin{lemma}Let $A$ be a separable, nuclear, and unital \textsl{C*}-algebraic quantum group, with faithful left Haar state $g$.    There exists a $\C$-linear map  $V\colon {  A\tensor A}\arrow A\otimes\H$ such that
$V(\coproduct{}(a)x)=aV(x),$ for all $a\in A$ and $x\in A\tensor A.$  This map extends to an ordinary unitary on the enveloping Hilbert spaces with  respect to the faithful states $g\tensor g$ on $A\tensor A$ as a Hilbert module over itself, and $ g $ on $A\tensor\H$ viewed as a Hilbert $A$-module.
\label{lem:special.module.coproduct.case}\end{lemma}
\begin{proof}Recall that $A\tensor A\subseteq A \tensor \H,$ as subsets. (This inclusion is not an inclusion of Hilbert modules because $A$ acts differently on the two sides of the inclusion.) From \cite[Prop. 3.6]{BS} or from \cite[Prop 3.21]{KustermansVaes} a multiplicative unitary is in fact an element of $\Mult(A\tensor\compact(\H)).$  Thus it multiplies $A\tensor \H$ into itself.  Then we have a well-defined   $\C$-linear map $V\colon A\tensor A \arrow A\tensor\H$ given by  $x\mapsto W^{*} x(1\tensor e),$ where $W$ is the above multiplicative unitary, and $eE$ is the image $e:=\pi_{\haar}1$ of $1$ in $\H.$ The property $V(\coproduct{}(a)x)=aV(x)$ is clear.
\end{proof}
 See Lemma \ref{lem:special.module.coaction.case} on page \pageref{lem:special.module.coaction.case} for a more constructive and explicit proof of the above.

\begin{definition} If $H_1$ and $H_2$ are Hilbert sub-modules of $A$, we denote by $H_1\boxproduct H_2$ the Hilbert sub-module of $A\tensor \H$ obtained by closing
$V(H_1\tensor H_2)$ in the Hilbert module norm on $A\tensor\H.$
\label{def:boxproduct}\end{definition}

An operator $V$ having the property $V(\coproduct{}(a)x)=aV(x)$ will be referred to as having the \textit{skew-linear} property. The skew-linear property implies ordinary linearity over $\C,$ because the coproduct $\coproduct{}$ is a unital homomorphism and thus maps $\lambda\Id_A$ to $\lambda\Id_{A\tensor A}.$

 We will see that there exists therefore a product on the Cuntz semigroup.

In the  case of stable rank 1, which is all we need here, we may as well define the Cuntz semigroup of a \textsl{C*}-algebra as being given by the isomorphism classes of the countably generated Hilbert modules over that algebra, together with direct sum as a semigroup operation, and inclusion as a relation giving an order structure on that semigroup \cite{CEI}. See \cite{CEI} for the adaptations that would be needed in the case of higher stable rank.

We now verify that the above definition \ref{def:boxproduct} does respect the equivalence relation on Hilbert modules in the Cuntz semigroup.
\begin{theorem} In a separable, nuclear, unital, and  stable rank 1 \textsl{C*}-algebraic quantum group $A,$ if $H_1,$ $H_2$ and $H_3$ are Hilbert sub-modules of $A$, and $H_1$ is isomorphic to $H_2,$ then $H_1\boxproduct H_3$ is isomorphic   to  $H_2\boxproduct H_3.$
\label{th:invariant}
\end{theorem}
\begin{proof}  We consider $W:=T(H_1\tensor H_3)$ which is a pre-Hilbert sub-module of $A\tensor\H$, and  $W':=T(H_2\tensor H_3),$ a pre-Hilbert sub-module of $A\tensor\H.$   We have, by hypothesis, that
$H_1\tensor H_3$ is isomorphic as a Hilbert $(A\tensor A)$-module to $H_2\tensor H_3.$ Denoting this isomorphism by $F\colon H_1\tensor H_3\arrow H_2\tensor H_3,$ we compose this isomorphism with the skew-linear mappings $T\inv$ and $T$ of Lemma \ref{lem:special.module.coproduct.case},  obtaining $J:=T\compose F\compose T\inv\colon W\arrow W'.$

 Thus, the map $J$ is, or more accurately extends to, a Hilbert space unitary. It is a module map by the skew-linear property of $T$ and the $(A\tensor A$-linearity of $F.$  By Proposition \ref{prop:2isometries.want.to.be.isomorphisms},  the map $J$ is isometric with respect to the norm coming from the enveloping Hilbert module $A\tensor\H$.  If we take the closure of  ${T(H_1\tensor H_3 )}$ and  ${T(H_1\tensor H_3 )}$ with respect to the Hilbert module norm on  $A\tensor\H,$ the fact that $J$ is an isometry allows us to extend $J$ by continuity to a bounded mapping, indeed, an isometry of the closures. We now have an isometry $J\colon\overline{T(H_1\tensor H_2 )} \arrow  \overline{T(H_1\tensor H_3 )}.$   Theorem \ref{th:equivalences.of.maps} then gives a Hilbert module isomorphism of these Hilbert $A$-modules.
\end{proof} \label{section.that.the.product.is.in}

\begin{corollary} There is an associative product $\boxproduct$ on the Cuntz semigroup of a stable rank 1 separable compact \textsl{C*}-algebraic quantum group. \label{cor:extend.product}\end{corollary}
\begin{proof} Our  proof above   only considered Hilbert C*-modules that are sub-modules of $A.$ Tensoring both the range and domain spaces of the unitary $V$ by a copy of the classic Hilbert space $\H,$ and using the fact that $\H\tensor\H\isom\H,$ extends the product operation to sub-modules of $\H_A.$ Then, just as in theorem \ref{th:invariant} we have that the operation $\boxproduct$ is well-defined up to isomorphism.

 Associativity with respect to multiplication follows  from co-associativity of the co-product.
\end{proof}

\section{Computing the product}

Compact \textsl{C*}-algebraic quantum groups have a well-behaved Fourier transform, which can be most briefly defined  by, following  Van Daele \cite{VanDaele1994}, see also  \cite{kahng}:
$$\beta(a,\F(b))=\haar(ab),$$ where  the elements $a$ and $b$ belong to a \textsl{C*}-algebraic quantum group $A,$ $\haar$ is the left Haar weight,  and $\beta(\cdot\,,\cdot)$ is the pairing with the dual algebra.
Following \cite[Def. 3.10]{kahng}, an operator-valued convolution product, $\convolve,$ can then be defined by the  property $\F(a\convolve b)=\F(a)\F(b),$ where $a$ and $b$ are elements of a \textsl{C*}-algebraic quantum group $A,$ and $\F$ is the Fourier transform defined previously.

This convolution operation takes, as is well-known, a pair of positive operators to a positive operator \cite[Theorem 1.3.3.i]{EnockSchwartz}.  One could wonder if the product defined in the previous section could, at the level of generators, be viewed as some kind of generalized convolution. We now show that this is indeed the case, at least under the technical condition, related to Kac algebras
\footnote{Conventional transliteration of Cyrillic suggests the spelling Kats algebra, however, the spelling Kac has become standard.},
  of having a tracial Haar state.

 We now assume tracial Haar state for the next Proposition only.  In the statement of the next Proposition, $\haar$ denotes the extension of the Haar state to the canonical trace related to the GNS Hilbert space associated with the Haar state.

In Theorem \ref{th:invariant}, if we take the product of two Hilbert modules $\overline{\ell_1 A}$ and $\overline{\ell_2 A},$ the product module will be some submodule of the standard Hilbert module $ H_A.$ By Cohen's theorem \cite{cohen} we expect this submodule to be singly generated as a Hilbert module, and we may ask if there is some nice description of some particular choice of that generator.

\begin{proposition} Let $A$ be a compact \textsl{C*}-algebraic quantum group with faithful Haar state.
Let $\ell_1$ and $\ell_2$ be in $A^{+}.$

The product module  $[\ell_1]\boxproduct[\ell_2]$ from Theorem \ref{th:invariant} has a (single) generator of the form
$$V(\ell_1\tensor \ell_2)V\inv\colon A\tensor\H\arrow A\tensor\H.$$

 If the Haar state is tracial,  $(\Id\tensor t) (\ell_1\boxproduct \ell_2)=\ell_1 \convolve \ell_2$ for all  $\ell_i\in A^{+},$ where  $t$ is the usual trace on $\compact.$\label{prop:box.and.convolve}
\end{proposition}
\begin{proof} (Outline) Since ${\ell_1 A} \boxproduct {\ell_2 A}$ is defined to be the closure of $V(\ell_1\tensor\ell_2)(A\tensor A),$ and since $V\inv (A\tensor\H)=A\tensor A$, we have that  ${\ell_1} \boxproduct {\ell_2}$ is the Hilbert module closure of the pre-Hilbert module $V(\ell_1\tensor\ell_2)V\inv (A\tensor\H).$

It is known \cite{fima} that we can find a representation where the tracial Haar state coincides with the usual (unbounded) trace on $\BH. $
Restricting $V(\ell_1\tensor \ell_2)V\inv$ to a map from $A\tensor\H$ to $A\tensor \H,$ and applying  $\Id\tensor t,$ we will consider the Fourier transform of $C:=(\Id\tensor t)V(\ell_1\tensor \ell_2)V\inv.$
We next note that  $\F(C)=\F(\ell_1)\F(\ell_2).$ This is a routine calculation using the definition of the Fourier transform in terms of the tracial Haar state, the fact that $t\tensor t$ is a trace, and the skew-linear property of $V.$

Thus we have shown that $t(aC)=\beta(a,\F(\ell_1)\F(\ell_2)),$ and since $t(aC)$ is equal to $\beta(a,\F(C)),$ the nondegeneracy of the pairing implies that $\F(C)=\F(\ell_1)\F(\ell_2).$ Thus, $C=\ell_1 \convolve \ell_2,$ where $C=(\Id\tensor t) (\ell_1\boxproduct \ell_2).$
\end{proof}

\section{Implications for K-theory}

In the case of stable rank one, everything comes together very nicely, as follows.
Clearly,  we can consider the equivalence classes of the projective modules within the Cuntz semigroup. It is known that equivalence of projective modules in the Cuntz semigroup is isomorphic to the ordinary equivalence of operator projections,  in matrix algebras over $A$. We note that by the first part of Proposition \ref{prop:box.and.convolve},  the product $\boxproduct$ will in fact take a pair of operator projections to a projection.
 Thus, the product $\boxproduct$ gives a product on  the semigroup of equivalence classes of projections, $V(A).$

 Stable rank one implies cancellation \cite[Prop. 6.5.1]{blackadar} in this semigroup of projections. K-theory is defined, in the unital case, as the enveloping  Grothendieck group of this semigroup $V(A)$ of projections.
Products on semigroups do not always give products on the enveloping Grothendieck group, but will do so exactly in the case that the semigroup is cancellative \cite{Dale}.
\begin{corollary} In a unital and separable stable rank 1 \textsl{C*}-algebraic quantum group, the product $\boxproduct$ provides a product on the K-group  $K_0 (A)$.
\end{corollary}

\begin{example} Recall that elements of $K_0$ are by definition formal differences of projections. We can use  K-theoretical techniques to reduce to the case where the element subtracted is of the form $[I_n]$ in some matrix algebra $M_n(A)$. Then, we can simplify our product operation to show that
\begin{equation}([p_1]-[I_n])\boxproduct([p_2]-[I_m])=[p_1\boxproduct p_2]-[I_k].
\label{eq:naive.product}\end{equation}
\end{example}
In fact, equation \eqref{eq:naive.product} seems to let us define a product on K-theory even without stable rank one.

The existence of the product then can be used to generalize our earlier results for the discrete case \cite{kuce.classify.Hopf.by.Ktheory}, with a real rank zero condition and the existence of the product replacing discreteness.

\section{Products and  extensions}

The topic of Hopf ideals is mostly of interest in the finite-dimensional case, but we attempt to state the definitions and first result more generally. A Hopf ideal of a Hopf algebra is an algebra ideal in the kernel of the co-unit that is also a co-algebra co-ideal and is stable under the antipode. Equivalently, a Hopf ideal is the intersection of the kernel of the co-unit with a normal and unital Hopf subalgebra. We make the assumption that the antipode is bijective, and hence that left normality equals right normality.

We also point out that in  general,  we may  obtain from our product construction two  possibly inequivalent product structures, one associated with the left Haar state and the other associated with the right Haar state.

\begin{theorem} Let $B$ be a \textsl{C*}-algebraic quantum group with faithful Haar state, separable, nuclear, and of stable rank 1 as a C*-algebra. Let $N$ be a normal Hopf sub-C*-algebra, and let $N^{+}$ denote its associated Hopf ideal. Then there is a short exact sequence
\begin{diagram}[small] 0&\rTo&\Cu  (N^{+})& \rTo & \Cu  (B) &\rTo & \Cu  (B/N^{+})&\rTo&0 \end{diagram}
where $\Cu  (B)$ is a semiring, and $\Cu  (N^{+})$ is an  ideal in that semiring. \label{th:ses}
\end{theorem}
\begin{proof} We note, first of all, that $N^{+}$ is a C*-algebra, and that there is a short exact sequence at the C*-algebra level,
\begin{diagram}[small] 0&\rTo&N^{+}& \rTo^\iota & B &\rTo^\pi & B/N^{+}&\rTo&0 \end{diagram}
The C*-algebraic exactness of the Cuntz semigroup functor was shown in \cite{ses}.
The connecting maps in the above sequence induce maps on Hilbert C*-modules that
are compatible with the Cuntz semigroup order relation, and that there is, at the level of ordered semigroups, a short exact sequence,
\begin{diagram}[LaTeXeqno] 0&\rTo&\Cu  (N^{+})& \rTo^{\Cu  (\iota)} & \Cu  (B) &\rTo^{\Cu  (\pi)} & \Cu  (B/N^{+})&\rTo&0. \label{diag:W}\end{diagram}

Since the co-product on $B/N^{+}$ is in fact the coproduct inherited from $B,$ applied to cosets rather than elements, it follows that $\Cu (\pi)$ is multiplicative. Thus, the kernel of $\Cu (\pi)$ is a semiring ideal (see \cite{Dale} for more information on semiring ideals) and an order ideal.
Since the sequence (\ref{diag:W}) is exact at the semigroup level, the image of the semigroup map $\Cu  (\iota)$ equals the kernel of $\Cu  (\pi).$
The image and the kernel do not change if we add multiplicative structure to the range and domain semigroups. In the semiring setting, if range equals kernel then we have exactness.\footnote{In the general case of exact sequences of semigroups, more needs to be checked, but we may as well use the additional structure at hand. In fact, the main way in which the semiring property is used in our proof is to conclude that a statement on range and kernel does imply exactness. 
} Thus, we have exactness, and it follows from this that the image of $\Cu  (\iota)$ is in fact an ideal in $\Cu  (B).$  Technically speaking, we have yet to define the product on  $\Cu  (N^{+}),$ since this was not previously discussed, but it is natural and appropriate to define it as the product pulled back from  $\Cu (B)$ by means of the map $\iota.$
\end{proof}

The main examples of Hopf ideals seem to be in the finite-dimensional case, and in this case, as is well-known, algebraic modules are automatically projective and the Cuntz semigroup in essence becomes $K$-theory. For example, see \cite{kuce.classify.Hopf.by.Ktheory0} for more information on the $K$-theory of finite-dimensional \textsl{C*}-algebraic quantum groups.
If we consider the finite-dimensional case, then:
 \begin{corollary} Let $B$ be a finite-dimensional \textsl{C*}-algebraic quantum group. Let $N$ be a normal Hopf sub-C*-algebra, and let $N^{+}$ denote its associated Hopf ideal. Then there is a short exact sequence
\begin{diagram} 0&\rTo&K_0 (N^{+})& \rTo & K_0 (B) &\rTo & K_0 (B/N^{+})&\rTo&0 \end{diagram}
where $K_0 (B)$ is an unital ring, and $K_0 (N^{+})$ is a ring ideal.
\end{corollary}
The above is possibly self-evident, but it is interesting to  obtain it as a special case of  a more general result.

The reason why the above is possibly self-evident is as follows: since the product of Hilbert modules that we have defined was constructed to be an analytical version of the algebraist's operation of restriction of rings (by the co-product homomorphism), being moreover compatible with the inner products that a Hilbert module must necessarily have, one would expect that in the finite-dimensional case, and indeed in the AF case, the closure operation that appears in our definition of the product should have no real significance.

In these cases, we could presumably have worked entirely with algebraic modules and the algebraic restriction of rings operation, making use of the already mentioned fact that over such an algebra, finitely generated algebraic modules are automatically projective.

\section{Co-actions}

Suppose $(A,\varDelta)$ is a compact or discrete quantum group. A left
 coaction $\delta$ of $(A,\varDelta)$ on a C$^*$-algebra $B$
is a $*$-homomorphism $\delta\colon B\to M(A\otimes B)$ such that
$\delta (B)(A\otimes 1)$  is dense
in $A\otimes B$  and that
$(\Id\otimes\delta)\delta =(\varDelta\otimes\Id)\delta$.
For example, the coproduct homomorphism $\varDelta$ can be viewed as a left coaction of a compact quantum group $(A, \varDelta)$ on $A$
and $\widehat{\varDelta}$ as a right coaction of $(\widehat A,
\widehat{\varDelta})$ on $\hat A$.
Indeed, following \cite[pg. 8]{martinboundary}, using a multiplicative unitary $W$ we may by the formulas
$$
\Phi (x)=W^* (1\otimes x)W\ \ \ \hbox{and}\ \ \
\hat\Phi (a)=W(a\otimes 1)W^*,$$
for $x\in\hat A$ and $a\in A$, define a left coaction $\Phi$ of
$(A,\varDelta )$ on $\hat A$ and a right coaction $\hat\Phi$ of $(\hat
A,\hat\varDelta )$ on $A$.

 In this example, the co-action homomorphisms are injective, but that need not be the case in general.

An invariant state for a left co-action $\delta$ of a compact or discrete \textsl{C*}-algebraic quantum group $(A,\coproduct{} )$ on $B$ is a state $\eta$ on $B$ such that $(\Id\tensor\eta)\delta=\eta(\cdot)1.$ For the next lemma, we assume the existence of such a faithful invariant state, this is a strong but natural assumption.

\begin{lemma}Let $A$ be a separable, unital, nuclear \textsl{C*}-algebraic quantum group, with faithful right Haar state $f$. Let $B$ be a unital \textsl{C*}-algebra. Suppose that there is a left co-action $\delta\colon B\to A\otimes B$ with faithful invariant state $\eta.$   There exists a $\C$-linear map  $W\colon \H_B\rightarrow  {A\tensor B}$ such that
$W(bx)=\delta (b)W(x)$ for all $x\in \H_B$ and $b\in B.$  This map extends to an ordinary unitary on the enveloping Hilbert spaces with  respect to the faithful states $\eta\tensor f$ on $A\tensor B$ as a Hilbert module over itself, and $\eta$ on the Hilbert $B$-module $B\tensor\H.$
\label{lem:special.module.coaction.case}\end{lemma}
\begin{proof} Let us first define a map $W\colon \H_B \rightarrow {A \tensor B},$
where  $\H_B$ is the standard Hilbert module over $B.$ We view the Hilbert module $\H_B:=B\tensor\H,$ as infinite sequences $(b_i)$ of elements of $B.$
Since $A$ has a faithful  state, we may embed $A$ in a Hilbert space  by the GNS construction. Taking a countable dense subset consisting of algebraic elements of $A$ and applying the Gram-Schmidt process at the Hilbert space level we obtain a countable basis $(g_i)$ of elements of $A$ for the Hilbert space.

Define a map $W$ from $\H_B$ into $A\tensor B$ by $W\colon (b_i)\mapsto \sum \delta(b_i)( g_i\tensor1 ).$ This map has dense domain because the domain contains at least the finite Hilbert modules $B\tensor\C^n.$ It is clear that the domain of this map is a $B$-module under the diagonal action of $B$ and that $W(bx)=\delta (b)W(x).$ The range is norm-dense in $A\tensor B$ because of the density property $\overline{\delta (B)(A\otimes 1)}=A\otimes B.$

 We now show that $W$ extends to a continuous map on the enveloping Hilbert spaces. Let the sequences $(b_i)$ and $(b'_j)$ denote  elements in $B\tensor\H.$ We recall that the inner product on $A\tensor B$ viewed as a Hilbert module over itself is $\ip{a}{b} =a^* b,$ and so if we take the inner product of $W(b_i)$ and $W(b'_j),$  we have
 \begin{equation} \begin{split}\ip{W((b_i))}{W((b'_j))}&= \left(\sum_{ij} (g_i^* \tensor1)\delta(b_i^*b'_j)(g_j\tensor1)\right)\\
\end{split}\label{eqn:basis}\end{equation}
where we used the fact that the coaction is a *-homomorphism, so that $\delta(b_i)^*\delta(b'_j)=\delta(b_i^*b'_j)$ .
	Applying  $\eta\tensor f$ to both sides of the above, and using the property $\eta(\delta(b_i^* b'_j))=\eta(b_i^* b'_j)$ we have:
	\begin{equation*}(\eta\tensor f)\ip{W((b_i))}{W((b'_j))}=\sum_{ij} \eta(b_i^*b'_j)f(g_{i}^* g_j).\end{equation*}
	It follows from the basis property of the $g_i$ that the right hand side simplifies to $\sum_{i} \eta(b_i^*b'_i)$ which can be rewritten as  $\eta$ applied to the usual inner product of $(b_i)$ and $(b'_j)$  in $\H_B,$ in other words, $\eta\left(\ip{(b_i)}{(b'_i)}_{B\tensor\H}\right).$
But then this shows that $(\eta\tensor f)\ip{W((b_i))}{W((b'_j))}=\eta\left(\ip{(b_i)}{(b'_i)}_{B\tensor\H}\right),$ which is the hypothesis needed to use Proposition \ref{prop:2isometries.want.to.be.isomorphisms}. Thus $W$ extends to an isometry on the enveloping Hilbert spaces, and since it will still have dense range  there, it is therefore a Hilbert space unitary.
\end{proof}

From the above lemma, we thus obtain a unitary $W\colon \H_B\rightarrow A\tensor B,$ at the level of Hilbert space structure, possessing a skew-linear property.
Tensoring on both sides with the classical Hilbert space \H, we can replace $A\tensor B$ with $\H_{A\tensor B}$. Just as in Theorem \ref{th:invariant}
we obtain, in the stable rank one case, a product $\boxproduct \colon \Cu (A) \times \Cu (B) \rightarrow \Cu(B).$

We can of course consider the special case where the co-action homomorphism is given by the co-product homomorphism, in other words, the case $A=B,$
and then we recover the theory discussed in section \ref{section.that.the.product.is.in}.
The associativity type property of the co-action, $$(\Id\otimes\delta)\delta =(\varDelta\otimes\Id)\delta,$$ implies that   the product induced in this way  from a co-product is compatible with the product induced from a co-action. In other words, products of the form $A_1 \boxproduct A_2 \boxproduct B_1,$ where $A_i$ is a Hilbert module over $A$ and $B_1$ is a Hilbert module over $B,$ are associative and we do not need to write parentheses. (c.f. Theorem \ref{cor:extend.product}.)
Rephrasing, what we have shown is that the Cuntz semigroup functor takes a co-action to a multiplicative action:
\begin{theorem} Let the \textsl{C*}-algebraic quantum group $A$ have a left co-action upon $B$ with a faithful invariant state. Then we obtain a well-behaved action of the semiring $\Cu(A)$ upon the semigroup $\Cu(B).$\end{theorem}

From a desire to simply illustrate the ideas, we have given detailed proofs for the case of stable rank 1 only, but the method of \cite{DKAxiomsPreprint} removes  this restriction.
In fact, if we consider Cuntz semigroups without restricting to stable rank 1, this amounts to using an equivalence relation coarser than just isomorphism of Hilbert modules, and it can be shown that the product we have defined is in fact still well-defined with respect to this coarser equivalence relation. Thus we claim that the above Theorem holds for higher stable rank as well.

 We omit the details because Cuntz semigroups for the case of higher stable rank are best studied using the rather technical open projection picture of the Cuntz semigroup. The open projection picture resembles K-theory sufficiently that  the techniques used \cite{kuce.classify.Hopf.by.Ktheory0} for classifying \textsl{C*}-algebraic quantum groups by K-theory may go through when K-theory is replaced by Cuntz semigroups.

It is quite possible that the above can be viewed as an algebraic topology type of obstacle to the existence of certain co-actions. There is considerable evidence that the Cuntz semigroup and sometimes K-theory contain a surprising amount of information about a \textsl{C*}-algebra. This being the case, one could expect the above invariant to also contain a lot of information. Certainly, the above can be used to provide algebraic obstacles to the existence of certain co-actions, and when co-actions do exist, it could be reasonable to in some sense classify them by their invariants in K-theory or the Cuntz semigroup.

The work of Goswami \cite{Goswami} on non-existence of certain co-actions, again in the compact case, provides evidence that the question of nonexistence may be especially interesting in the compact case that we have been looking at. Eventually, applications may lie in the direction of the noncommutative Borsuk-Ulam theorem \cite{BDH2015}.

%
 \section*{Conflict of interest}

 The author declares  no conflict of interest.



\end{document}